\documentclass{article}
\usepackage[utf8]{inputenc}
\usepackage[margin=1in]{geometry}
\usepackage[titletoc,title]{appendix}
\usepackage{amsmath,amsfonts,amssymb,mathtools}
\usepackage{float}
\usepackage{amsthm}
\usepackage{enumitem}

\usepackage{titlesec}
\usepackage{sectsty}
\titleformat*{\section}{\LARGE\bfseries}
\titleformat*{\subsection}{\Large\bfseries}
\titleformat*{\subsubsection}{\large\bfseries}
\sectionfont{\Huge}
\subsectionfont{\large}
\subsubsectionfont{\normalsize}
\usepackage{multicol}
\usepackage{lmodern}
\usepackage{marvosym} 
\usepackage{lipsum}
\usepackage{mwe}
\usepackage{caption}
\usepackage{subcaption} 
\newtheorem{definition}{Definition}[section]
\newtheorem{example}{Example}[section]
\newtheorem{theorem}{Theorem}[section]

\newtheorem{lemma}{Lemma}[section]

\newtheorem{remark}{Remark}[section]

\begin{document}

\title{ On some properties of $\omega$-uniqueness in tensor complementarity problem }

\author{A. Dutta$^{a,1}$, R. Deb$^{a,2}$, A. K. Das$^{b,3}$\\
\emph{\small $^{a}$Jadavpur University, Kolkata , 700 032, India.}\\	
\emph{\small $^{b}$Indian Statistical Institute, 203 B. T.
	Road, Kolkata, 700 108, India.}\\
%\emph{\small $^{1}$Email: rwitamjanaju@gmail.com}\\
\emph{\small $^{1}$Email: aritradutta001@gmail.com}\\
\emph{\small $^{2}$Email: rony.knc.ju@gmail.com}\\
\emph{\small $^{3}$Email: akdas@isical.ac.in}\\
}

\date{}

\maketitle

\begin{abstract}
\noindent In this article we introduce column adequate tensor in the context of tensor complementarity problem and consider some important properties. The tensor complementarity problem is a class of nonlinear complematarity problems with the involved function being defined by a tensor. We establish the inheritance property and invariant property of column adequate tensors. We show that TCP$(q,\mathcal{A})$ has $\omega$-unique solution under some assumptions.
\\


\end{abstract}
\footnotetext[1]{Corresponding author}

\section{Introduction}
The tensor complementarity problem is a class of nonlinear complematarity problems with the involved function being defined by a tensor, which is also direct and natural extension of the linear complementarity problem. During last several years, the tensor complementarity problem attains much attraction and has been studied extensively with respect to theory, to solution methods and applications. In recent years, various tensors with special structures have been studied. For details, see \cite{qi2017tensor} and \cite{song2015properties}. The tensor complementarity problem was studied initially by Song and Qi \cite{song2014properties}.
The tensor complementarity problem is a subclass of the non-linear complementarity problems where the function involved in the non-linear complementarity problem is a special polynomial defined by a tensor in the tensor complementarity problem. The polynomial functions used in tensor complementarity problems have some special structures.
\\
 For a given mapping $F: \mathbb{R}^n \mapsto \mathbb{R}^n$ the complementarity problem is to find a vector $x\in \mathbb{R}^n $ such that
\begin{equation}\label{Classical comp problem}
   x\geq 0, ~~F(x) \geq 0, ~~ \mbox{and}~ x^{T}F(x)=0.
\end{equation}
If $F$ is nonlinear mapping, then the problem (\ref{Classical comp problem}) is called a nonlinear complementarity problem \cite{facchinei2007finite}, and if $F$ is linear function, then the problem (\ref{Classical comp problem}) reduces to a linear complementarity problem \cite{cottle2009linear}. The linear complementarity problem may be defined as follows:\\
Given a matrix $ M\in \mathbb{R}^{n\times n} $ and a vector $q \in \mathbb{R}^n$, the linear complementarity problem \cite{cottle2009linear}, denoted by $LCP(q,M)$, is to find a pair of vectors $w,z \in\mathbb{R}^n$ such that 
\begin{equation}\label{linear complementarity problem}
    z\geq 0, ~~~ w=Mz+q\geq 0, ~~~ z^T w=0.
\end{equation}
A pair of vectors $(w,z)$ satisfying (\ref{linear complementarity problem}) is called a solution of the $LCP(q,M)$.  A vector $z$ is called a $z$-solution if there exists a vector $w$ such that $(w,z)$ is a solution of the $LCP(q,M)$. Similarly vector $w$ is called a $w$-solution if there exists a vector $z$ such that $(w,z)$ is a solution of the $LCP(q,M)$. The solution $(w, z)$ is said to be $w$-unique if all solution pairs $(w, x)$, have the same $w$, for a given $q$. Ingleton \cite{ingleton1966probelm}, \cite{ingleton1970linear} studied the $w$-uniqueness of solutions to linear complementarity problem.

\noindent The algorithm based on principal pivot transforms namely Lemke's algorithm, Criss-cross algorithm which are used to find the solutions of linear complementarity problem are studied extensively considering several matrix classes. For details, see \cite{jana2018processability}, \cite{das2017finiteness}, \cite{ neogy2005principal}, \cite{das2016properties}, \cite{neogy2012generalized}, \cite{jana2019hidden}, \cite{jana2021more}, \cite{mohan2001more}. In this context properties of matrix classes ensure the processability of $LCP(A,q)$. For details see \cite{neogy2006some}, \cite{ neogy2013weak}, \cite{ neogy2005almost}, \cite{ neogy2011singular}, \cite{neogy2005principal}, \cite{neogy2006some}. Complementarity theory helps to establish the connection between KKT point and the optimal point for structured stochastic game and multi-objective program problem. For details see \cite{mondal2016discounted}, \cite{ neogy2008mathematical}, \cite{ neogy2008mixture}, \cite{ neogy2005linear}, \cite{ neogy2016optimization}, \cite{ mohan2004note}, \cite{ neogy2009modeling}.\\
Now we consider the case of $F(x)=\mathcal{A}x^{m-1} + q$ with $\mathcal{A}\in T_{m,n}$ and $ q\in \mathbb{R}^n$ then the problem (\ref{Classical comp problem}) becomes
\begin{equation}\label{ Tensor Complementarity problem}
    x\geq 0, ~~~\mathcal{A}x^{m-1} + q\geq 0, ~~~\mbox{and}~~ x^{T}(\mathcal{A}x^{m-1}+q)=0
\end{equation}
which is called a tensor complementarity problem, denoted by the $TCP(q,\mathcal{A})$.
Denote $\omega=\mathcal{A}x^{m-1} + q$, then the tensor complementarity problem is to find $x$ such that 
\begin{equation}\label{ Tensor Complementarity problem 2}
    x\geq 0, ~~~ \omega= \mathcal{A}x^{m-1} + q\geq 0, ~~~\mbox{and}~~ x^{T}\omega=0.
\end{equation}
A pair of vectors $(\omega, x)$ satisfying (\ref{ Tensor Complementarity problem 2}) is called a solution of the $TCP(q,\mathcal{A})$.  A vector $x$ is called a $x$-solution if there exists a vector $\omega$ such that $(\omega, x)$ is a solution of the $TCP(q,\mathcal{A})$. Similarly vector $\omega$ is called a $\omega$-solution if there exists a vector $x$ such that $(\omega, x)$ is a solution of the $TCP(q,\mathcal{A})$. The solution $(\omega, x)$ is said to be $\omega$-unique if all solution pairs $(\omega, x)$, have the same $\omega$, for a given $q$.

\noindent Motivated by the discussion on positive definiteness of multivariate homogeneous polynomial forms \cite{bose1976general}, \cite{hasan1996procedure}, \cite{jury1981positivity}, Qi \cite{qi2005eigenvalues} introduced the concept of symmetric positive definite (positive semi-definite) tensors. Song and Qi \cite{song2015properties} studied $P(P_0)$-tensors and $B(B_0)$-tensors.
The equivalence between (strictly) semi-positive tensors and (strictly) copositive tensors in symmetric case   were shown by Song and Qi \cite{song2016tensor}. The existence and uniqueness of solution of $TCP(q,\mathcal{A})$ with some special tensors were discussed by Che, Qi, Wei \cite{che2016positive}. The boundedness of the solution set of the $TCP(q,\mathcal{A})$ was studied by Song and Yu \cite{song2016properties}. The sparse solutions to $TCP(q,\mathcal{A})$ with a $Z$-tensor and its method to calculate were obtained by Luo, Qi and Xiu \cite{luo2017sparsest}. The equivalent conditions of solution to $TCP(q, \mathcal{A})$ were shown by Gowda, Luo, Qi and Xiu \cite{gowda2015z} for a $Z$-tensor $\mathcal{A}$. The global uniqueness of solution of $TCP(q, \mathcal{A})$ was considered by Bai, Huang and Wang \cite{bai2016global} for a strong $P$-tensor $\mathcal{A}$.
The properties of $TCP(q, \mathcal{A})$ was studied by Ding, Luo and Qi \cite{ding2018p} for a new class of $P$-tensor. The properties of the several classes of $Q$-tensors were presented by Suo and Wang \cite{huang2015q}. In this article we introduce column adequate tensor in the context of tensor complementarity problem and study different properties of this tensor.

The paper is organized as follows. Section 2 presents some basic notations and results. In section 3, we introduce the column adequate tensors and study tensor properties as well as the properties of $\omega$-solution of $TCP(q, \mathcal{A})$. In this context, we consider auxiliary matrix formation. We establish the inheritance property and invariant property of column adequate tensors.

\section{Preliminaries}

We begin by introducing some basic notations used in this paper. We consider tensor, matrices and vectors with real entries. Let $m$th order $n$ dimensional real tensor $\mathcal{A}= (a_{i_1 i_2 ... i_m}) $ be a multidimensional array of entries $a_{i_1 i_2 ... i_m} \in \mathbb{R}$ where $i_j \in [n]$ with $j\in [m]$. $T_{m,n}$ denotes the set of real tensors of order $m$ and dimension $n.$ Any $\mathcal{A}= (a_{i_1 i_2 ... i_m}) \in T_{m,n} $ is called a symmetric tensor, if the entries $a_{i_1 i_2 ... i_m}$ are invariant under any permutation of their indices. $S_{m,n}$ denotes the collection of all symmetric tensors of order $m$ and dimension $n$  where $m$ and $n$ are two given positive integers with $m,n\geq 2$. 
An identity tensor $\mathcal{I}m=(\delta_{i_1 ... i_m})\in T_{m,n}$ is defined as follows:
\[ \delta_{i_1 ... i_m}= \left\{
\begin{array}{ll}
	  1  &;\; i_1= ...= i_m, \\
	  0  &;\; else.
	   \end{array}
 \right. \]
For any positive integer $n,$  $[n]$ denotes set of $\{ 1,2,...,n \}$. All vectors are column vectors. Let $\mathbb{R}^n$ denote the $n$-dimensional Euclidean space and $\mathbb{R}^n_+ :=\{ x\in \mathbb{R}^n : x\geq 0 \}$. Let $\mathbb{Z}$ denote the set of integers and $\mathbb{Z}^n_+ :=\{ x\in \mathbb{Z}^n : x\geq 0 \}$. For $\alpha=(\alpha_1, \alpha_2,...,\alpha_n)\in \mathbb{Z}^n_+$, $|\alpha|$ be defined as $|\alpha| = \alpha_1 + \alpha_2 + \cdot \cdot \cdot +\alpha_n.$ For any $x \in \mathbb{R}^n$, let $x^{[m]}\in \mathbb{R}^n$ with its $i$th component being $x^m_i$ for all $i \in [n]$. For $\mathcal{A}\in T_{m,n} $ and $x\in \mathbb{R},\; \mathcal{A}x^{m-1}\in \mathbb{R}^n $ is a vector defined by
\[ (\mathcal{A}x^{m-1})_i = \sum_{i_2, i_3, ...i_m =1}^{n} a_{i i_2 i_3 ...i_m} x_{i_2} x_{i_3} \cdot\cdot \cdot x_{i_m} , ~~~\mbox{for all}~i \in [n] \]
and $\mathcal{A}x^m\in \mathbb{R} $ is a scalar defined by
\[ \mathcal{A}x^m = \sum_{i_1,i_2, i_3, ...i_m =1}^{n} a_{i_1 i_2 i_3 ...i_m} x_{i_1} x_{i_2} \cdot\cdot \cdot x_{i_m}. \]

\noindent The general product of tensors was introduced by Shao (\cite{shao2013general}). Let $\mathcal{A}$ and $\mathcal{B}$ be order $m \geq 2$ and order $k \geq 1$, $n$-dimensional tensor, respectively. The product $\mathcal{A} \cdot \mathcal{B}$ is a tensor $\mathcal{C}$ of order $(m-1)(k-1) + 1$ and $n$-dimensional with entries 
\[c_{i \alpha_1 \cdots \alpha_{m-1} } =\sum_{i_2, \cdots ,i_m \in[n]} a_{i i_2 \cdots i_m} b_{i_2 \alpha_1} \cdots b_{i_m \alpha_{m-1}},\] where $i \in [n]$, $\alpha_1, \cdots, \alpha_{m-1} \in [n]^{k-1}$. %This product was proved to be associative by Shao (\cite{shao2013general}).

\noindent Let $\mathbb{R}[x_1, . . . , x_n]$ denotes the set of all polynomials in $x_1, \; x_2\; ...\; x_n$ with coefficients in $\mathbb{R}$. Any monomial in $\mathbb{R}[x_1, . . . , x_n]$ is denoted by $x^{\alpha}=x^{\alpha_1}_1 \cdots x^{\alpha_n}_n ,$ where the $n$-tuple of exponents $\alpha=(\alpha_1 \cdots \alpha_n) \in \mathbb{Z}^n_+ $. There is a one-to-one correspondence between the monomials in $\mathbb{R}[x_1, . . . , x_n]$ and $\mathbb{Z}^n_+$. Furthermore, any ordering $>$ on the space $\mathbb{Z}^n_+$ provide an ordering on monomials i.e. if $\alpha > \beta $ according to this ordering implies $x^{\alpha} > x^{\beta} $. Now we consider some definitions and results which are used in the next section. 

\begin{definition}\cite{cox2013ideals}
(Lexicographic Order). Let $\alpha = (\alpha_1, \cdots, \alpha_n)$ and $\beta = (\beta_1, \cdots, \beta_n)$ be in $\mathbb{Z}^n_+$. We say $\alpha >_{lex} \beta$ if the leftmost nonzero entry of the vector difference $\alpha -\beta \in \mathbb{Z}^n $ is positive. We write $x^{\alpha} >_{lex} x^{\beta} $ if $\alpha >_{lex} \beta$.
\end{definition}
\noindent Example 1. Let $\alpha = (1, 2, 0)$ and $\beta = (0, 3, 4) $. Then $(1, 2, 0) >_{lex} (0, 3, 4)$, since $\alpha -\beta = (1,-1,-4)$.

\noindent Example 2. Let $\alpha = (3, 2, 4) $ and $\beta = (3, 2, 1) $. Then $(3, 2, 4) >_{lex} (3, 2, 1)$, since $\alpha -\beta = (0, 0, 3)$.

\noindent Example 3. Naturally $x_1 >_{lex} x_2 >_{lex} \cdots >_{lex} x_n$.

\begin{definition}\cite{cox2013ideals}
(Graded Lexicographic Order).  Let $\alpha, \beta \in \mathbb{Z}^n_+$. We say $\alpha >_{grlex} \beta$  if
\[|\alpha|= \sum_{i=1}^{n} \alpha_i > |\beta|= \sum_{i=1}^{n} \beta_i, \mbox{ or } |\alpha|= |\beta| \mbox{ and } \alpha >_{lex} \beta .\]
\end{definition}

\noindent Example 1. Let $\alpha = (1, 2, 3) $ and $\beta = (3, 2, 0) $. Then $(1, 2, 3) >_{grlex} (3, 2, 0)$, since $|(1, 2, 3)| = 6 > |(3, 2, 0)| = 5$.

\noindent Example 2. Let $\alpha = (1, 2, 4) $ and $\beta = (1, 1, 5) $. Then $(1, 2, 4) >_{grlex} (1, 1, 5)$, since $|(1, 2, 4)| = |(1, 1, 5)|$ and $(1, 2, 4) >_{lex} (1, 1, 5)$.

\noindent Example 3. Consider the monomials $x_1^2,\; x_2^2,\; x_3^2,\; x_1x_2,\; x_2x_3,\; x_3x_1 $. Then in grlex order they are arranged as 
    \[ x_1^2 >_{grlex} x_1x_2 >_{grlex} x_1x_3 >_{grlex} x_2^2 >_{grlex} x_2x_3 >_{grlex} x_3^2 .\]

\begin{definition} \cite{ingleton1966probelm}
A matrix $M\in \mathbb{R}^{n\times n}$ is said to be column adequate matrix if 
\[ z_i (Mz)_i \leq 0 \; \; \mbox{for all } i=1,2,...,n \;\; \implies \; Mz=0.  \]
\end{definition}

\begin{definition}\cite{song2015properties}
Let $\mathcal{A}= (a_{i_1 i_2 ...i_m}) \in T_{m,n}$. Let $J\subseteq [n]$ with $|J|=r,\; 1\leq r \leq n$. Then a principal sub-tensor of $\mathcal{A}$ is denoted by $\mathcal{A}^J_r$ and is defined as
\[ \mathcal{A}^J_r = ( a_{i_1 i_2 ...i_m} ), \; \forall \; i_1, i_2,...i_m \in J .\]
\end{definition}
\begin{example}
Let $\mathcal{A}\in T_{3,3}$ be such that $a_{111}=2,\; a_{222}=-1,\; a_{333}=2,\; a_{223}=-2,\; a_{232}=1,\; a_{233}=-1 $ and all other entries of $\mathcal{A}$ are zeros.

\noindent For $r=1$, there are three principal sub-tensors of one dimension. Let $J_1 = \{1\}$, $J_2 = \{2\}$ and $J_3 = \{3\}$. Then  principal sub-tensors of one dimension are $\mathcal{A}^{J_{1}}_1 = (a_{111})=(2)$, $\mathcal{A}^{J_{2}}_1 = (a_{222})=(-1)$, $\mathcal{A}^{J_{3}}_1 = (a_{333})=(3)$.

\noindent For $r=2$, there are three principal sub-tensors of two dimension. Let $J_4 = \{1,2\}$, $J_5 = \{2,3\}$, $J_6 = \{1,3\}$. The principal sub-tensors of two dimension are $\mathcal{A}^{J_{4}}_2, \; \mathcal{A}^{J_{5}}_2, \; \mathcal{A}^{J_{6}}_2$. Here 
$ \mathcal{A}^{J_{4}}_2 = (a_{ijk}), \; \forall \; i,j,k \in J_{4}.$
Then $\mathcal{A}^{J_{4}}_2$ is given by $a_{111}=2,\; a_{222}=-1 $ and all other entries of $\mathcal{A}^{J_{4}}_2$ are zeros. Here 
$ \mathcal{A}^{J_{5}}_2 = (a_{ijk}), \; \forall \; i,j,k \in J_{5}.$
Then $\mathcal{A}^{J_{5}}_2$ is given by $a_{222}=-1,\; a_{333}=2,\; a_{223}=-2,\; a_{232}=1,\; a_{233}=-1 $ and all other entries of $\mathcal{A}^{J_{5}}_2$ are zeros. Here 
$ \mathcal{A}^{J_{6}}_2 = (a_{ijk}), \; \forall \; i,j,k \in J_{6}$.
Then $\mathcal{A}^{J_{6}}_2$ is given by $ a_{222}=-1, \; a_{333}=2 $ and all other entries of $\mathcal{A}^{J_{6}}_2$ are zeros.

\noindent For $r=3$, there is only one principal sub-tensor of three dimension which coincides with $\mathcal{A}$.
\end{example}

\begin{definition}
\cite{song2016properties} Given $\mathcal{A}= (a_{i_1 i_2 ... i_m}) \in T_{m,n} $ and $q\in \mathbb{R}^n$, a vector $x$ is said to be (strictly) feasible, if $x(>) \geq 0$ and $\mathcal{A}x^{m-1}+q (>) \geq 0$.

\noindent $TCP(q,\mathcal{A})$, defined by equation (\ref{ Tensor Complementarity problem}) is said to be (strictly) feasible if a (strictly) feasible vector exists. 

\noindent $TCP(q,\mathcal{A})$ is said to be solvable if there is a feasible vector $x$ satisfying $x^{T}(\mathcal{A}x^{m-1}+q)=0$ and $x$ is the solution.
\end{definition}

\begin{definition}
\cite{qi2005eigenvalues} A tensor $\mathcal{A}= (a_{i_1 i_2 ... i_m}) \in T_{m,n} $ is said to be a positive definite (positive semi-definite), if $\mathcal{A}x^m > (\geq)0 $ for all $x\in \mathbb{R}^n \backslash \{0\}$. Furthermore if $\mathcal{A}\in S_{m,n}$ and satisfies the above condition then $\mathcal{A}$ is said to be a symmetric positive definite (positive semi-definite) tensor.
\end{definition}

\begin{definition}
\cite{song2014properties} A tensor $\mathcal{A}= (a_{i_1 i_2 ... i_m}) \in T_{m,n} $ is said to be (strictly) semi-positive tensor if for each $x\in \mathbb{R}^n_+ \backslash \{0\}$, there exists an index $i\in [n]$ such that $x_i>0$ and $(\mathcal{A}x^{m-1})_i (>)\geq 0$.
\end{definition}

\begin{definition}
\cite{song2015properties} A tensor $\mathcal{A}= (a_{i_1 i_2 ... i_m}) \in T_{m,n} $ is said to be a $P(P_0)$-tensor, if for each $x\in \mathbb{R}^n \backslash \{0\}$, there exists an index $i\in [n]$ such that $x_i \neq 0$ and $x_i (\mathcal{A}x^{m-1})_i > (\geq 0)$.
\end{definition}

\begin{definition}\cite{bai2016global}
A tensor $\mathcal{A}\in T_{m,n}$ is said to be strong $P$-tensor if for any two different $x=(x_i)$ and $y=(y_i)$ in $\mathbb{R}^n$, $max_{1 \leq i \leq n} (x_i - y_i)(\mathcal{A}x^{m-1} - \mathcal{A}y^{m-1} )_i >0$.
\end{definition}

\begin{definition}
\cite{huang2015q} A tensor $\mathcal{A}= (a_{i_1 i_2 ... i_m}) \in T_{m,n} $ is said to be a $Q$-tensor if the $TCP(q,\mathcal{A})$ (\ref{ Tensor Complementarity problem}) is solvable for all $q\in \mathbb{R}^n $.
\end{definition}

\begin{definition}
\cite{chen2018column} A tensor $\mathcal{A}= (a_{i_1 i_2 ... i_m}) \in T_{m,n} $ is said to be a column sufficient tensor, if for $x \in \mathbb{R}^n$, $x_i(\mathcal{A}x^{m-1})_i \leq 0,\; \forall \; i\in [n] ~\implies$ $x_i(\mathcal{A}x^{m-1})_i =0, \; \forall \; i\in [n]$.
\end{definition}

\begin{definition}\cite{huang2015q}
A tensor $\mathcal{A}\in T_{m,n}$ is said to be strong $P_0$-tensor if for any $x=(x_i)$ and $y=(y_i)$ in $\mathbb{R}^n$ with $x \neq y$, $max_{1 \leq i \leq n} (x_i - y_i)(\mathcal{A}x^{m-1} - \mathcal{A}y^{m-1} )_i \geq0$. It is abbreviated as $SP_0$-tensor. The set of all real $m$-order $n$-dimensional $SP_0$-tensors are denoted by $T_{m,n}\cap SP_0$.
\end{definition}

\begin{definition}
\cite{song2015properties} A tensor $\mathcal{A}= (a_{i_1 i_2 ... i_m}) \in T_{m,n} $ is said to be weak $P(P_0)$-tensor, if for each $x\in \mathbb{R}^n \backslash {0}$, there exists an index $i\in [n]$ such that $x_i \neq 0$ and $x_i^{m-1} (\mathcal{A}x^{m-1})_i > (\geq 0)$.
\end{definition}
\begin{remark}
For even order tensors the set of $P_0$-tensors and the set of weak $P_0$-tensors are same.
\end{remark}

\begin{definition}\cite{shao2016some}
Let $\mathcal{A} \in T_{m,n}$ and $R_i(\mathcal{A})= (r_{i i_2 ...i_m })_{i_2 ...i_m}^n \in T_{m-1,n}$ such that $r_{i i_2 ... i_m} = a_{i i_2 ...i_m} $, then $\mathcal{A}$ is called row-subtensor diagonal, or simply row diagonal, if all its row-subtensors $R_i(\mathcal{A}), \; i=1,2,...n$ are diagonal tensors, namely, if $a_{i i_2 ...i_m}$ can take nonzero values only when $i_2 = \cdot \cdot \cdot = i_m$.
\end{definition}
\begin{example}
Let $\mathcal{A}\in T_{4,2}$ such that $a_{1111}=3,\; a_{2222}=1,\; a_{1222}=-2,\; a_{2111}=1\;$ and all other entries of $\mathcal{A}$ are zeros. Then $\mathcal{A}$ is row-diagonal tensor.
\end{example}

\begin{definition}\cite{pearson2010essentially}
Let $\mathcal{A} \in T_{m,n}$. The majorization matrix $M(\mathcal{A})$ of $\mathcal{A}$ is $n\times n$ matrix with entries $M(\mathcal{A})_{i j} = a_{ijj...j}$, where $i,j=1,2,...,n $.
\end{definition}
\begin{example}
Let $\mathcal{A}\in T_{4,2}$, be such that $a_{1111}= 1,\;a_{1222}= -1,\; a_{2112}= 3,\; a_{2111}= 0,\; a_{1112}= -2,\; a_{2222}= 2\;$ and all other entries of $\mathcal{A}$ are zeros. Then $M(\mathcal{A})=\left( \begin{array}{cc}
    1 & -1\\
    0 & 2
\end{array} \right)$.
\end{example}
\begin{theorem}\cite{shao2016some}
Let $\mathcal{A}$ be an order $m$ and dimension $n$ tensor. Then $\mathcal{A}$ is row diagonal if and only if $\mathcal{A}=M(\mathcal{A})\mathcal{I}_m,$ where $\mathcal{I}_m$ is the identity tensor of order $m$ and dimension $n$.
\end{theorem}

\begin{theorem}\cite{bai2016global}
For any $q\in \mathbb{R}^n$ and a $P$-tensor $\mathcal{A} \in T_{m,n}$, the solution set of $TCP(q,\mathcal{A})$ is nonempty and compact.
\end{theorem}

\begin{theorem}\cite{bai2016global}
Let $\mathcal{A}\in T_{m,n}$ be a strong $P$-tensor. Then the $TCP(q, \mathcal{A})$ has the property of global uniqueness and solvability property.
\end{theorem}

\begin{theorem}\cite{huang2015q}
If $\mathcal{A}\in T_{m,n} \cap SP_0$, we have 
\[ \mathcal{A}\in R_0 \iff \mathcal{A}\in R \iff \mathcal{A}\in ER \iff \mathcal{A}\in Q .\]
\end{theorem}

\begin{theorem}\cite{chen2018column}
Let $\mathcal{A}$ is a column sufficient tensor with order $m$ dimension $n$. Then the $TCP(q,\mathcal{A})$ has unique zero solution for every $q \geq 0$.
\end{theorem}

\begin{theorem}\cite{ingleton1966probelm,ingleton1970linear}\label{matrix w uniqueness result}
Let $A$ be a matrix of order $n$. Then the following conditions are equivalent
\begin{enumerate}[label=(\alph*)]
    \item $A$ is column adequate matrix.
    \item $LCP(q, A)$ has $w$-unique solution.
\end{enumerate}
\end{theorem}

\section{Main results}
We begin by the definition of column adequate tensor.

\begin{definition}
A tensor $\mathcal{A} \in T_{m,n} $ is said to be column adequate tensor if $x_i (\mathcal{A}x^{m-1})_i \leq 0 , ~\forall \; i \in [n] $ implies $\mathcal{A}x^{m-1} =0$.
\end{definition}

\begin{remark}
Let $\mathcal{A}$ be a $P$-tensor. Then the condition $ x_i (\mathcal{A}x^{m-1})_i \leq 0 \;\forall \; i \in [n]$ holds only for $x=0.$ Therefore a $P$-tensor is trivially a column adequate tensor.
% and so is column competent.
\end{remark}

\begin{example}
Let $\mathcal{A}\in T_{4,2},$ such that $a_{1 1 1 1}=2, ~a_{1 1 1 2}=1, ~a_{2 1 2 2}=4, ~a_{2 2 2 2}=2$ and all other entries of $\mathcal{A}$ are zeros.

Then $\mathcal{A}x^3=
\left(
	\begin{array}{l}
	 (2x_1 + x_2)x_1^2 \\
	 2(2x_1 + x_2)x_2^2
	 	\end{array}
\right) .$ Let $x_i (\mathcal{A}x^3)_i \leq 0$ for $i=1,2.$ i.e. $(2x_1 +x_2)x_1^3 \leq 0$ and $(2x_1 +x_2)2x_2^3 \leq 0$.

\noindent Now $(2x_1 +x_2)x_1^3 < 0 \implies \left\{
\begin{array}{ll}
	 \mbox{either,} & x_1<0 \mbox{ and } 2x_1 +x_2 >0, \mbox{ implies } (2x_1 +x_2)x_2^3 >0, \\
	  \mbox{or,} & x_1>0 \mbox{ and } 2x_1 +x_2 <0, \mbox{ implies } (2x_1 +x_2)x_2^3 >0.
	   \end{array}
 \right.$
Similarly $(2x_1 +x_2)x_2^3 < 0 \implies (2x_1 +x_2)x_1^3 >0.$
Therefore the systems of inequations
\begin{equation}
\begin{split}
(2x_1 +x_2)x_1^3 & < 0 \\ 
(2x_1 +x_2)2x_2^3 & \leq 0
\end{split}
\end{equation}
and \begin{equation}
\begin{split}
(2x_1 +x_2)x_1^3 & \leq 0 \\ 
(2x_1 +x_2)2x_2^3 & < 0
\end{split}
\end{equation} 
are inconsistent. Now these two equations \begin{equation}
\begin{split}
(2x_1 +x_2)x_1^3 & = 0 \\ 
(2x_1 +x_2)2x_2^3 & = 0
\end{split}
\end{equation} 
have solution if $(2x_1 +x_2)=0.$ This implies $\mathcal{A}x^3=
\left(
	\begin{array}{l}
	 0\\
	 0
	 	\end{array}
\right).$ Therefore $\mathcal{A}$ is column adequate tensor.

\noindent To show that $\mathcal{A}$ is neither a $P$-tensor nor a $PSD$ tensor, consider $x=\left( \begin{array}{c}
    1 \\
    -1.5
\end{array} \right).$
\end{example}

\begin{example}
%A PSD tensor can be a column adequate tensor. 
Let $\mathcal{A} \in T_{4,2} ,$ such that $a_{1111}=1, ~a_{1112}=-1, ~a_{2111}=1$ and all other entries of $\mathcal{A}$ are zeros.
\noindent Then $\mathcal{A}x^4 = x_1^4 \geq 0.$ Therefore $\mathcal{A}$ is a non-symmetric PSD tensor. Again 
$\mathcal{A}x^3=
\left(
	\begin{array}{c}
	 x_1^2(x_1 -x_2) \\
	 x_1^3
	 	\end{array}
\right).$ 
Then we have $x_1 (\mathcal{A}x^3)_1 = x_1^3(x_1-x_2),$ and $x_2 (\mathcal{A}x^3)_2 = x_1^3 x_2.$ Now $x_i (\mathcal{A}x^3)_i \leq 0$ implies $x_1^3(x_1-x_2) \leq 0, $ and $x_1^3 x_2 \leq 0 .$ Now it is easy to check that if $x_1 (\mathcal{A}x^3)_1 < 0$ then $x_2 (\mathcal{A}x^3)_2 > 0$ and if $x_2 (\mathcal{A}x^3)_2 < 0$ then $x_1 (\mathcal{A}x^3)_1 > 0$. Solving $x_i (\mathcal{A}x^3)_i =0, \; \forall \; i\in [2]$ we obtain  $x=\left( \begin{array}{c}
    0 \\
    k
\end{array} \right) $ where $k \in \mathbb{R}.$ For $x=\left( \begin{array}{c}
    0 \\
    k
\end{array} \right)$ we obtain $\mathcal{A}x^3 =\left( \begin{array}{c}
    0 \\
    0
\end{array} \right).$ Hence $\mathcal{A}$ is column adequate tensor.
\end{example}

\begin{remark}
Adequate tensors are such tensors that every vector whose sign is reversed by the function $F(x) = \mathcal{A}x^{m-1}$ belongs to the kernel of the function $F(x) = \mathcal{A}x^{m-1}$.
\end{remark}

\begin{theorem}
A column adequate tensor is a column sufficient tensor.
\end{theorem}
\begin{proof}
Let $\mathcal{A} \in T_{m,n} $ be column adequate tensor. Then for $x\in \mathbb{R}^n, \; x_i(\mathcal{A}x^{m-1} )_i \leq 0$, $\forall \; i \in [n] \implies \mathcal{A}x^{m-1}=0 \implies x_i (\mathcal{A}x^{m-1} )_i =0, \; \forall \; i\in [n] $. This implies that $\mathcal{A}$ is column sufficient tensor.  
\end{proof}

\begin{remark}
Not all column sufficient tensors are column adequate tensor. Below we give an example of a column sufficient tensor which is not column adequate tensor.
\end{remark}

\begin{example}
Consider the tensor $\mathcal{A} \in T_{4,2}$, such that  
$ a_{1112}=-2, ~a_{2111}= a_{2222}=1$ and all other entries of $\mathcal{A}$ are zeros. Then $\mathcal{A}$ is shown to be a a column sufficient tensor in \cite{chen2018column}. 
Then
$\mathcal{A}x^{3}=
\left(
	\begin{array}{c}
	  -2 x_1^2 x_2 \\
	   x_1^3 +x_2^3
	 	\end{array}
\right) $
 and $x_1 (\mathcal{A}x^3)_1 = -2 x_1^3 x_2$, $ x_2 (\mathcal{A}x^3)_2 = x_1^3 x_2 + x_2^4 $. Then for $x=\left( \begin{array}{c}
    1 \\
    0
\end{array} \right)$ we have $x_i (\mathcal{A}x^3)_i = 0, \; \forall \; i \in [2]$ but $\mathcal{A}x^3 =\left( \begin{array}{c}
    0 \\
    1
\end{array} \right) \neq \left( \begin{array}{c}
    0 \\
    0
\end{array} \right).$
 Therefore $\mathcal{A}$ is not a column adequate tensor.
 \end{example}

In the following result we prove the inheritance property of column adequate tensors in the principal sub-tensor point of view.

\begin{theorem}
Suppose that  $\mathcal{A} \in T_{m,n} $ is a column adequate tensor then all principal sub-tensors of $\mathcal{A}$ are column adequate tensors.
\end{theorem}
\begin{proof}
Let $J \subseteq [n]$, $|J|= r$ and $\mathcal{A}_r^J$ be the principal sub-tensor of  $\mathcal{A}$. For some $x\in \mathbb{R}^r$, if $x_i(\mathcal{A}_r^J x^{m-1})_i \leq 0,~\forall \; i\in J $, we construct $y\in \mathbb{R}^n $ such that $y_i= \left\{
\begin{array}{ll}
	  x_i  &;\; \forall \; i\in J \\
	  0  &; \; \forall \; i\in J^c
	   \end{array}
 \right.$.
Then we have $(\mathcal{A}y^{m-1})_i=(\mathcal{A}_r^J x^{m-1})_i, \; \forall \; i\in J$, which follows  \[y_i(\mathcal{A}y^{m-1})_i = \left\{
\begin{array}{cc}
	  x_i(\mathcal{A}_r^J x^{m-1})_i \leq 0 & ; \; \forall \; i\in J \\
	  0   & ; \; \forall \; i\in J^c
	   \end{array}
 \right. .\]
Thus $y_i(\mathcal{A}y^{m-1})_i \leq 0, \; \forall \; i\in[n] $. Since $\mathcal{A}$ is column adequate tensor, $y_i(\mathcal{A}y^{m-1})_i \leq 0, \; \forall \; i\in[n] \implies (\mathcal{A}y^{m-1})_i=0 ~, \forall \; i \in [n] $ $\implies (\mathcal{A}_r^J x^{m-1})_i=0, \; \forall \; i \in J $.
Hence $\mathcal{A}_r^J$ is a column adequate tensor.
\end{proof}

\begin{theorem}
Let $\mathcal{A} \in T_{m,n}$ and $P=diag(p_1, p_2, ..., p_n)$, $Q=diag(q_1, q_2, ..., q_n) $ be two diagonal matrices of order $n$ where $p_i q_i >0, ~\forall \; i \in [n]$. Then $ \mathcal{A}$ is column adequate tensor if and only if $ P\mathcal{A}Q $ is column adequate tensor.
\end{theorem}
\begin{proof}
Let $ P= diag(p_1, p_2, ..., p_n)$ and $Q= diag(q_1, q_2, ..., q_n)$ be two diagonal matrices of order $n$ where $p_i q_i >0, ~\forall \; i \in [n]$. Then for $P=[p_{ij}]_{n \times n}$, $p_{ij}=\left \{ \begin{array}{ll}
	  p_i  &;\; \forall \; i=j \\
	  0  &; \; \forall \; i \neq j
	   \end{array}  \right.$, and 
$Q=[q_{ij}]_{n \times n}$, $q_{ij}=\left \{ \begin{array}{ll}
	  p_i  &;\; \forall \; i=j \\
	  0  &; \; \forall \; i\neq j
	   \end{array}  \right.$, with $p_i q_i >0, ~\forall \; i \in [n]$.
Let $\mathcal{A} = (a_{i_1 i_2 ... i_m})\in T_{m,n}$ be column adequate tensor with order $m$ and dimension $n$. Let $\mathcal{B}= P \mathcal{A} Q$, assume $\mathcal{B} = (b_{i_1 i_2 ... i_m})$, then by the definition of multiplication for all $i_1, i_2, ...,i_m \in [n]$, we have,
\[ b_{i_1 i_2 ...i_m} = \sum_{j_1, j_2, ...,j_m \in [n] } p_{i_1 j_1} a_{j_1 j_2 ...j_m} q_{j_2 i_2} q_{j_3 i_3} \cdot \cdot \cdot q_{j_m i_m}\]
\[ = p_{i_1} a_{i_1 i_2 ...i_m} q_{ i_2} q_{ i_3} \cdot \cdot \cdot q_{ i_m}  \]
Then for $x\in \mathbb{R}^n$ and $i\in [n], $ we obtain, 
\begin{align*}
   x_i(\mathcal{B}x^{m-1})_i  & = x_i \sum_{ i_2, ...,i_m \in [n] } b_{i i_2 ...i_m} x_{ i_2} x_{ i_3} \cdot \cdot \cdot x_{ i_m} \\
      & = x_i \sum_{ i_2, ...,i_m \in [n] } p_i a_{i i_2 ...i_m} q_{i_2} q_{i_3} \cdot \cdot \cdot q_{i_m}   x_{ i_2} x_{ i_3} \cdot \cdot \cdot x_{ i_m} \\
       & = \frac{p_i}{q_i} (q_i x_i) \sum_{ i_2, ...,i_m \in [n] } a_{i i_2 ...i_m} (q_{i_2} x_{ i_2}) (q_{i_3}  x_{ i_3}) \cdot \cdot \cdot  (q_{i_m}x_{ i_m})\\
       & = \frac{p_i}{q_i} y_i \sum_{ i_2, ...,i_m \in [n] } a_{i i_2 ...i_m} y_{ i_2} y_{ i_3} \cdot \cdot \cdot y_{ i_m} \\
       & = \frac{p_i}{q_i} y_i(\mathcal{A}y^{m-1})_i 
\end{align*}

\noindent where $y=Q x$. Since $ p_i q_i>0, \; \forall \; i\in [n]$, $x_i(\mathcal{B}x^{m-1})_i \leq 0 \iff y_i(\mathcal{A}y^{m-1})_i \leq 0, \forall \; i \in [n]$.
Since $\mathcal{A} $ is column adequate tensor so $\forall \; i \in [n],$  $y_i(\mathcal{A}y^{m-1})_i \leq 0 \implies \mathcal{A}y^{m-1}= 0 $. Again $(\mathcal{B}x^{m-1})_i = p_i(\mathcal{A}y^{m-1})_i,~ \forall \; i\in [n] $. Therefore $\mathcal{A}y^{m-1}= 0 \implies \mathcal{B}x^{m-1}=0 $. Thus $x_i(\mathcal{B}x^{m-1})_i \leq 0, ~\forall \; i \in [n] \implies \mathcal{B}x^{m-1}=0 $. Hence, $\mathcal{B} = P \mathcal{A} Q$ is column adequate tensor.

Conversely, let $P \mathcal{A} Q$ be column adequate tensor. Since the entries of $P$ and $Q$ are such that $p_i q_i>0, ~\forall \; i\in [n]$, then $P$ and $Q$ are invertible with $P^{-1}= diag (\frac{1}{p_1}, \cdots , \frac{1}{p_n})$ and $Q^{-1}= diag (\frac{1}{q_1}, \cdots , \frac{1}{q_n})$, where $ \frac{1}{p_i} \cdot \frac{1}{p_i} = \frac{1}{p_i q_i}>0, ~\forall \; i\in [n]$. Therefore by the first part of the proof, the tensor $P^{-1} (P \mathcal{A} Q) Q^{-1} = \mathcal{A} $ %(by associativity of tensor multiplication)
is column adequate tensor.
\end{proof}

Now we prove invariant property of column adequate tensors under rearrangement of subscripts. 

\begin{theorem}
Let $\mathcal{A} \in T_{m,n}$ and $P =[p_{ij}]_{n \times n}$ be a permutation matrix of order $n$. Then $\mathcal{A}$ is column adequate tensor if and only if $P^T\mathcal{A} P$ is column adequate tensor.
\end{theorem}
\begin{proof}
Let $P$ be a permutation matrix and $\mathcal{A} \in T_{m,n}$ be a  column adequate tensor. Now consider the tensor $\mathcal{B}= P^T \mathcal{A} P.$ Then 
\[ b_{i_1 i_2 ... i_m} = \sum_{j_1, j_2, ..., j_m \in [n]} p_{j_1 i_1} a_{j_1 j_2 ...j_m} p_{j_2 i_2} p_{j_3 i_3} \cdots p_{j_m i_m}. \]
Since $P$ is a permutation matrix, only one entry in each row and column is one. Without loss of generality, for any $i\in [n],$ suppose $p_{i^{\prime}i}=1$ and $p_{ji}=0, \; j\neq i^{\prime}, \; j \in [n].$ Then we have 
\begin{equation*}
     b_{i_1 i_2 ... i_m} = \sum_{j_1, j_2, ..., j_m \in [n]} p_{j_1 i_1} a_{j_1 j_2 ...j_m} p_{j_2 i_2} p_{j_3 i_3} \cdots p_{j_m i_m}
\end{equation*}
\begin{equation}\label{permutation matrix multiplication}
     = p_{i_1^{\prime} i_1} a_{i_1^{\prime} i_2^{\prime} \dots i_m^{\prime}} p_{i_2^{\prime} i_2} p_{i_3^{\prime} i_3} \cdots p_{i_m^{\prime} i_m}.
\end{equation}
Consider $x\in \mathbb{R}^n$ such that $x_i(\mathcal{B}x^{m-1})_i \leq 0, \; \forall \; i\in [n].$ Then by $(\ref{permutation matrix multiplication}),$
\begin{align*}
     x_i(\mathcal{B}x^{m-1})_i  & = x_i \sum_{ i_2^{\prime}, ...,i_m^{\prime} \in [n]} p_{i^{\prime} i} a_{i^{\prime} i_2^{\prime} \dots i_m^{\prime}} p_{i_2^{\prime} i_2} p_{i_3^{\prime} i_3} \cdots p_{i_m^{\prime} i_m} x_{ i_2} x_{ i_3} \cdot \cdot \cdot x_{ i_m} \\
      & = p_{i^{\prime} i} x_i \sum_{ i_2^{\prime}, ...,i_m^{\prime} \in [n]}  a_{i^{\prime} i_2^{\prime} \dots i_m^{\prime}} p_{i_2^{\prime} i_2} p_{i_3^{\prime} i_3} \cdots p_{i_m^{\prime} i_m} x_{ i_2} x_{ i_3} \cdot \cdot \cdot x_{ i_m} \\
       & = y_{i^{\prime}} \sum_{ i_2^{\prime}, ...,i_m^{\prime} \in [n]}  a_{i^{\prime} i_2^{\prime} \dots i_m^{\prime}} y_{i_{2^{\prime}}} y_{i_{3^{\prime}}} \cdots y_{i_{m^{\prime}}} \\
       & = y_{i^{\prime}}(\mathcal{A}y^{m-1})_{i^{\prime}} \\
       & \leq 0.
\end{align*}
Then for $y=P x \in \mathbb{R}^n$ we have $x_i(\mathcal{B}x^{m-1})_i = y_{i^{\prime}}(\mathcal{A}y^{m-1})_{i^{\prime}} \leq 0, \; \forall \; i \in [n].$ This implies $\mathcal{A}y^{m-1}= 0$, since $\mathcal{A}$ is column adequate tensor. Then $(\mathcal{B}x^{m-1})_i = p_{i^{\prime}i}(\mathcal{A}y^{m-1})_{i^{\prime}} = 0,$ $\forall \; i\in [n],$ since $\mathcal{A}y^{m-1}= 0.$
Thus $x_i(\mathcal{B}x^{m-1})_i \leq 0, \; \forall \; i \in [n] \implies \mathcal{B}x^{m-1}=0 $. Hence, $\mathcal{B} = P^T \mathcal{A} P$ is column adequate tensor.

Conversely, Let $\mathcal{B}= P^T \mathcal{A} P$ be column adequate tensor. Since $P$ is a permutation matrix, then so is $P^{-1}$ and we have, $P^{-1}=P^T$ and $(P^{-1})^T=P$. Then by the first part of the proof, $(P^{-1})^T \mathcal{B} P^{-1}$ is column adequate tensor. Now $(P^{-1})^T \mathcal{B} P^{-1}= (P^T)^{-1} (P^T \mathcal{A} P) P^{-1}= \mathcal{A} $. % (by associativity of tensor multiplication).
Hence $\mathcal{A}$ is column adequate tensor.
\end{proof}

\begin{theorem}
Let $\mathcal{A} \in T_{m,n}$ be a column adequate tensor. Then $\mathcal{A}$ is a $P_0$-tensor.
\end{theorem}
\begin{proof}
Let $\mathcal{A} \in T_{m,n}$ be column adequate tensor. Then for $x\in \mathbb{R}^n, ~ x_i(\mathcal{A}x^{m-1})_i \leq 0, \;\forall \; i \in [n] $, implies $\mathcal{A}x^{m-1} = 0$. Now we show that for all $x\in \mathbb{R}^n \backslash \{0\}$, there exists an index $i\in [n]$ such that $x_i \neq 0$ and $x_i (\mathcal{A}x^{m-1})_i  \geq 0$. If possible, let $\exists \; y \in \mathbb{R}^n $ such that 
\begin{equation}\label{adequate to P0}
  y_i(\mathcal{A}y^{m-1})_i <0, \; \forall \; i \in [n].  
\end{equation}
Since $\mathcal{A}$ is column adequate tensor so by the inequalities (\ref{adequate to P0}) we conclude $\mathcal{A}y^{m-1} = 0$ which in turn $y_i(\mathcal{A}y^{m-1})_i=0 $. This contradicts (\ref{adequate to P0}). Hence $\mathcal{A} \in P_0$.
\end{proof}

\begin{remark}
Not all $P_0$-tensors are column adequate tensor.
\end{remark}
 Below we give an example of a $P_0$-tensor which is not column adequate tensor.
 
\begin{example}
Let $\mathcal{A} \in T_{4,2} $ such that $a_{1111}=1, ~a_{1222}=-1, ~a_{2221}=1$ and all other entries of $\mathcal{A}$ are zeros. Then $\mathcal{A}x^4 = x_1^4 \geq 0$, which implies $\mathcal{A}$ is a non-symmetric PSD tensor as well as $P_0$-tensor.

\noindent But $\mathcal{A}x^3=
\left(
	\begin{array}{c}
	 x_1^3 -x_2^3 \\
	 x_1 x_2^2
	 	\end{array}
\right)$
and $x_1 (\mathcal{A}x^3)_1 = x_1(x_1^3-x_2^3),$ $x_2 (\mathcal{A}x^3)_2 = x_1 x_2^3.$ For $x=\left( \begin{array}{c}
    0 \\
    k
\end{array} \right), ~k\neq 0 $ we have $x_1 (\mathcal{A}x^3)_1 = 0, ~x_2 (\mathcal{A}x^3)_2 =0 $ but $\mathcal{A}x^3 = \left( \begin{array}{c}
    -k^3 \\
    0
\end{array} \right) \neq \left( \begin{array}{c}
    0 \\
    0
\end{array} \right).$ Therefore $\mathcal{A}$ is not a column adequate tensor.
\end{example}

\begin{definition}
A tensor $\mathcal{A} \in T_{m,n} $ is said to be weak column adequate if $x_i^{m-1} (\mathcal{A}x^{m-1})_i \leq 0 , \; \forall \; i \in [n] $ implies $\mathcal{A}x^{m-1} =0.$
\end{definition}

\begin{example}
Let $\mathcal{A}\in T_{m,n}$, where $a_{111}=1$ and all other entries of $\mathcal{A}$ are zeros. Then $\mathcal{A}x^2 = \left( \begin{array}{c}
    x_1^2 \\
    0
\end{array} \right) .$ Now $x_i^2 (\mathcal{A}x^2)_i \leq 0, \; \forall \; i\in[2] $ implies $x=\left( \begin{array}{c}
    0 \\
    0
\end{array} \right) .$ For which $\mathcal{A}x^2= \left( \begin{array}{c}
    0 \\
    0
\end{array} \right).$ Therefore $\mathcal{A}$ is a weak column adequate tensor. To show that $\mathcal{A}$ is not column adequate consider $x=\left( \begin{array}{c}
    -1\\
    0
\end{array} \right) .$
\end{example}

\begin{remark}
For even order tensor the set of weak column adequate tensor is equal to the set of column adequate tensor.
\end{remark}
\begin{theorem}
Let $\mathcal{A} \in T_{m,n}$ be weak column adequate tensor. Then $\mathcal{A}$ is weak $P_0$-tensor.
\end{theorem}
\begin{proof}
Let $\mathcal{A} \in T_{m,n}$ be weak column adequate tensor. Then for $x\in \mathbb{R}^n, ~ x_i^{m-1}(\mathcal{A}x^{m-1})_i \leq 0, \;\forall \; i \in [n] $, implies $\mathcal{A}x^{m-1} = 0$. Now we show that for all $x\in \mathbb{R}^n \backslash \{0\}$, there exists an index $i\in [n]$ such that $x_i \neq 0$ and $x_i (\mathcal{A}x^{m-1})_i  \geq 0$. If possible, let $\exists \; y \in \mathbb{R}^n $ such that 
\begin{equation}\label{weak adequate to weak P0}
  y_i^{m-1}(\mathcal{A}y^{m-1})_i <0, \; \forall \; i \in [n].
\end{equation}
Since $\mathcal{A}$ is column adequate tensor so by inequalities (\ref{weak adequate to weak P0}) we conclude $\mathcal{A}y^{m-1} = 0$ which in turn $y_i^{m-1}(\mathcal{A}y^{m-1})_i=0 $. This contradicts (\ref{weak adequate to weak P0}). Hence $\mathcal{A}$ is a weak $P_0$-tensor.
\end{proof}

\subsection{Auxiliary matrix formation}
Consider a tensor of order $m$ and dimension $n$, that is $\mathcal{A} \in T_{m,n}$ . Then for any vector $x =(x_1, x_2, ... ,x_n)^T\in \mathbb{R}^n $ the equation $\mathcal{A}x^{m-1} = b$ gives a system of homogeneous polynomial equations in $n$ variables of degree $(m-1)$. Here the number of equations are $n$ and in each polynomial equation, the number of monomials are $N$ , where $N$ is equal to number of monomials in $x_1,\;  x_2,\; ... ,\;x_n ~$ of degree $(m-1),$ i.e. $N= \binom{m+n-2}{n-1}$. Let $\mathcal{A} = (a_{i_1 i_2 ... i_n}) \in T_{m,n}$. Then $(\mathcal{A}x^{m-1})_i$ is a homogeneous polynomial of degree $(m-1)$ in $n$ variables $x_1, x_2, ... ,x_n$, which can be represented as,
\begin{equation*}
\begin{split}
 (\mathcal{A}x^{m-1})_i & =  \sum_{i_2, i_3, ..., i_n =1}^n a_{i i_2 ... i_n} x_{i_2}x_{i_3}\cdot \cdot \cdot x_{i_n} \\
  & =  \sum_{\alpha_1 + \alpha_2 + \cdot \cdot \cdot +\alpha_n =m-1} C^i_{(\alpha_1, \alpha_2, \cdot \cdot \cdot ,\alpha_n)}   x_1^{\alpha_1} x_2^{\alpha_2} \cdot \cdot \cdot x_n^{\alpha_n}\\
  & = \sum_{|\alpha|= m-1} C_{\alpha}^i x^{\alpha} 
 \end{split}
\end{equation*}

\noindent where $\alpha_i \in [m-1] \cup \{0\}, ~i=1,2,...n$ and $\alpha =(\alpha_1, \alpha_2, \cdot \cdot \cdot ,\alpha_n)$ such that $|\alpha| = \alpha_1 + \alpha_2 + \cdot \cdot \cdot +\alpha_n =m-1 $.

Here $x^{\alpha} = x_1^{\alpha_1} x_2^{\alpha_2} \cdot \cdot \cdot x_n^{\alpha_n} $ and 
\begin{equation*}
\begin{split}
 C_{(\alpha_1, \alpha_2, \cdot \cdot \cdot ,\alpha_n)}^i & = C_{\alpha}^i \\
 & = \mbox{Coefficient of the monomial } x^{\alpha} \mbox{ in } (\mathcal{A}x^{m-1})_i \\
 & = \mbox{Coefficient of the monomial } x_1^{\alpha_1} x_2^{\alpha_2} \cdot \cdot \cdot x_n^{\alpha_n} \mbox{ in } (\mathcal{A}x^{m-1})_i
\end{split}
\end{equation*}

\noindent Clearly $C_{\alpha}^i =$ Sum of all the elements of the form $a_{i \pi \{ (\alpha_1 ~times~ 1)(\alpha_2 ~times~ 2)\cdot \cdot \cdot (\alpha_n ~times~ n) \}}$ from $\mathcal{A}$, where $\pi \{ (\alpha_1 ~times~ 1)(\alpha_2 ~times~ 2)\cdot \cdot \cdot (\alpha_n ~times~ n) \}$ represents all possible permutation of $1,1,... \alpha_1 \mbox{ times },2,2,... \alpha_2 \mbox{ times },... n,n,...\alpha_n \mbox{ times }$ and if for any $k\in [n], \alpha_k =0 $, then $0$ times $k$ has no impact on $C_\alpha^i$.

\noindent Now we define modified graded lexicographic order. 
\begin{definition}
 Let $\alpha, \beta \in \mathbb{Z}^n_+$. We say $\alpha >_{mglo} \beta$ if any one of the followings hold:
\begin{enumerate}
    \item For $\alpha$ such that $\alpha_i = \left\{
          \begin{array}{ll}
           	  |\alpha|  &;\; \forall \; i\in [n] \\
        	  0  &; \; otherwise.
     	   \end{array}
       \right.$ and for $\beta$ such that $\beta_i = \left\{
          \begin{array}{ll}
           	  |\beta|  &;\; \forall \; i\in [n] \\
        	  0  &; \; otherwise.
     	   \end{array}
       \right.$ and $\alpha >_{lex} \beta$.
    \item For $\alpha$ such that $\alpha_i = \left\{
          \begin{array}{ll}
           	  |\alpha|  &;\; \forall \; i\in [n] \\
        	  0  &; \; otherwise.
     	   \end{array}
       \right.$ and for $\beta$ such that $\beta $ contains atleast two non-zero components. % and $\alpha >_{grlex} \beta$.
    \item For $\alpha, \beta$ such that $\alpha, \beta$ contains at least two non-zero components and $\alpha >_{grlex} \beta$.
\end{enumerate}

\end{definition}
\begin{example}
\begin{enumerate}
    \item Let $\alpha = (2, 0, 0)$ and $\beta = (0, 2, 0).$ Then $(2, 0, 0) >_{mglo} (0, 2, 0),$ since $(2, 0, 0) >_{grlex} (0, 2, 0)$.
    \item Let $\alpha = (2, 0, 0)$ and $\beta = (2, 1, 0).$ Then $(2, 0, 0) >_{mglo} (2, 1, 0),$ since $(2, 1, 0)$ has two nonzero components and $(2, 0, 0)$ has only one non-zero component, though $(2, 1, 0) >_{grlex} (2, 0, 0) .$
    \item Let $\alpha = (2, 2, 0)$ and $\beta = (2, 1, 1).$ Then $(2, 2, 0) >_{mglo} (2, 1, 1),$ since $(2, 2, 0) >_{grlex} (2, 1, 1).$
    \item Consider the monomials $x_1^2,\; x_2^2,\; x_3^2,\; x_1x_2,\; x_2x_3,\; x_3x_1 $. Then in modified graded lex order they are arranged as 
    \[ x_1^2 >_{mglo} \; x_2^2 >_{mglo} \; x_3^2 >_{mglo} \; x_1x_2 >_{mglo} \; x_1x_3  >_{mglo} \; x_2x_3 .\]
\end{enumerate}
\end{example}

\subsubsection{Auxiliary matrix}
\noindent The polynomials involved in $\mathcal{A}x^{m-1}$ are elements of the vector space of homogeneous polyniomials of degree $(m-1)$ in $n$ variables $x_1, \; x_2, \; ...,\; x_n $. A basis of this vector space is $ \{x_1^{m-1},\; \cdots,\; x_n^{m-1},\; x_1^{m-2}x_2,\; \cdots,\; x_1^{m-2}x_n,\; \cdots,\; x_1x_n^{m-2} \}$ which contains $N$ number of monomials as its basis vectors. Consider the ordered basis $B$ of the space of homogeneous polynomials with modified graded lexicographic order as $B=\{ x_1^{m-1},\; x_2^{m-1},\; \cdots, x_n^{m-1},\; x_1^{m-2}x_2,$ $\cdots ,\;x_1x_n^{m-2} \}$.

\noindent Now if we say $y_1 = x_1^{m-1},\; y_2= x_2^{m-1},\; \cdots,\; y_n = x_n^{m-1}, $ and assign $y_{n+1},\; y_{n+2},\; \cdots,\;y_N $ with the respective ordered vectors from $B$, then $B$ can be considered as $B^{\prime} = \{ y_1,\; y_2,\; \cdots,$ $y_n, y_{n+1},\; \cdots, y_N \}$. Let for some $j\in [N]$, $y_j = x^\alpha$, for the $j$-th monomial $x^\alpha$ in $B$. Then the polynomial $(\mathcal{A}x^{m-1})_i$ can be expressed as $\sum_{j=1}^N a_{i j} y_j$, where $a_{i j} = C_{\alpha}^i =$ coefficient of $j$-th monomial in $(\mathcal{A}x^{m-1})_i$. Then $\mathcal{A}x^{m-1} $ can be expresseed as $Ay$ where $y=( y_1,\; y_2,\; \cdots,\; y_n,\; y_{n+1},\; \cdots, y_N )$ and $A$ is the coefficient matrix $A = (a_{i j})_{n \times N} $. Therefore, $\mathcal{A}x^{m-1} = A y $.

\noindent Consider the tensor complementarity problem $TCP(q,\mathcal{A})$ which is to find, $x \in \mathbb{R}^n$ such that
\[x \geq 0, ~~~~ \omega= \mathcal{A}x^{m-1} +q \geq 0,~~~~ x^T \omega = 0.\]
Since $\mathcal{A}x^{m-1} = A y $, so with help of the matrix $A$, we define an auxiliary $LCP$. The matrix $A$ is a rectangular matrix of order $n \times N$, where $n <N $. We define a new square matrix say auxiliary matrix $\Bar{A}$ by conjoining $N-n$ number of zero rows below the matrix $A$, that is $\Bar{A}=\left(\begin{array}{c}
     A  \\
     O
\end{array} \right)$, where $O$ is a null matrix of order $(N-n) \times N $.

\noindent Let $\Bar{q} \in \mathbb{R}^N$ be defined as $\Bar{q}_i = \left\{
\begin{array}{ll}
	  q_i \; ; & \forall \; i\in [n] \\
	  0 \; ;  & \forall \; i\in [N]\backslash [n]
	   \end{array}
 \right.$,
i.e. $\Bar{q}=\left(\begin{array}{c}
     q  \\
     0
\end{array} \right),$ where $0$ is a null vector of order $(N-n)$.

\noindent Then the $LCP(\Bar{q},\Bar{A})$ is to find $ y\in \mathbb{R}^N $ such that,
\begin{equation}\label{auxiliary LCP}
    y \geq 0, \;\; w = \Bar{A}y + \Bar{q} \geq 0, \;\; y^T w =0,
\end{equation}
\[  \]
is said to be the auxiliary $LCP$ to the $TCP(q,\mathcal{A})$.
%\section{Results}

\noindent Now we establish some relation between solution of $TCP(q,\mathcal{A})$ and solution of $LCP(\Bar{q},\Bar{A})$.

\begin{theorem}\label{solution of LCP implies solution of TCP}
Consider the $TCP(q,\mathcal{A})$ where $\mathcal{A}\in T_{m,n}$ and for some $q\in \mathbb{R}^n.$ Let $LCP(\Bar{q},\Bar{A})$ be the corresponding auxiliary $LCP$ defined by equations (\ref{auxiliary LCP}) and $x=(x_1,\; x_2,\;$ $... ,\;x_n)^T\in SOL(q,\mathcal{A}) $. Construct the vector $y=(y_1, y_2, ..., y_n, y_{n+1},y_{n+2},..., y_N )^T$ such that for all $j\in [N]$, $y_j$ is associated with the monomial $x_1^{\alpha_1} x_2^{\alpha_2} \cdot \cdot \cdot x_n^{\alpha_n}$, $\alpha_1 + \cdots + \alpha_n = m-1 $, by modified graded lexicographic order.  Then $y$ is a solution of $LCP(\Bar{q},\Bar{A})$. That is $x \in  SOL(q,\mathcal{A}) \implies y\in SOL(\Bar{q},\Bar{A}) $
\end{theorem}
\begin{proof}
Let $x\in SOL(q,\mathcal{A}) $, then $x \geq 0, ~~ \omega= \mathcal{A}x^{m-1} +q \geq 0$, and $ x^T \omega = 0 $. Therefore for  $x\in SOL(q,\mathcal{A}) $,
\begin{equation}\label{component of x}
    x_i \geq 0,\; \forall \; i \in[n]
\end{equation} 
\begin{equation}\label{component of w}
    \omega_i = (\mathcal{A}x^{m-1})_i +q_i \geq 0, \; \forall \; i \in[n]
\end{equation} 
and
\begin{equation}\label{complementarity condition}
    x_i \omega_i = x_i ((\mathcal{A}x^{m-1})_i +q_i) =0,\; \forall \; i \in[n] .
\end{equation}
Now by the construction of the vector $y \in \mathbb{R}^N$, it is clear that,
\begin{equation}\label{component of y}
    y_i \geq 0, \; \forall \; i \in[N],\mbox{ by inequalities }(\ref{component of x}) .
\end{equation} 
Now $w_i = (\Bar{A}y)_i + \Bar{q}_i = \left\{
\begin{array}{ll}
	 (A y)_i + q_i  \; & ; \; \forall \; i\in [n] \\
	 0 \;  & ; \; \forall \; i\in [N]\backslash [n]
	   \end{array}
 \right. $. Since $A y = \mathcal{A}x^{m-1} $ which gives $ (A y)_i = (\mathcal{A}x^{m-1})_i ,\; \forall \; i\in [n] $. This implies 
 \[w_i = \left\{
\begin{array}{ll}
	 (\mathcal{A}x^{m-1})_i + q_i  \;  & ; \; \forall \; i\in [n] \\
	 0 \;   & ; \; \forall \; i\in [N]\backslash [n]
	   \end{array}
 \right. =  \left\{
\begin{array}{ll}
	 \omega_i  \;  & ; \; \forall \; i\in [n] \\
	 0 \;   & ; \; \forall \; i\in [N]\backslash [n]
	   \end{array}
 \right. .\]
\noindent Therefore by inequalities (\ref{component of w})
\begin{equation}\label{component of w bar}
    w_i = (\Bar{A}y)_i + \Bar{q}_i \geq 0, \; \forall \; i \in[N].
\end{equation}
\noindent Finally, \begin{align*}
     y_i w_i = y_i ((\Bar{A}y)_i + \Bar{q}_i) 
= & \left\{
\begin{array}{ll}
	 y_i ((Ay)_i + q_i)  \;  & ; \; \forall \; i\in [n] \\
	 0 \;   & ; \; \forall \; i\in [N]\backslash [n]
	   \end{array}
 \right.\\
 = & \left\{
\begin{array}{ll}
	 x_i^{m-1}((\mathcal{A}x^{m-1})_i +q_i)  \;  & ; \; \forall \; i\in [n], \mbox{ since } y_i = x_i^{m-1}, \; \forall \; i \in [n] \\
	 0 \;   & ; \; \forall \; i\in [N]\backslash [n]
	   \end{array}
 \right..
 \end{align*}
Therefore by equations (\ref{complementarity condition}) we obtain,
\begin{equation}\label{complementarity condition of conv prob}
    y_i w_i = y_i ((\Bar{A}y)_i + \Bar{q}_i) = 0 ,\; \forall \; i \in[N].
\end{equation}
Hence by equations (\ref{component of y}), (\ref{component of w bar}) and (\ref{complementarity condition of conv prob}) we conclude $y \in SOL(\Bar{q},\Bar{A}) $.

\end{proof}

We illustrate Theorem 3.7 with the help of two examples.

\begin{example}
Let $\mathcal{A} \in T_{3,2}$ such that $a_{111}=a_{222}=1 $ and all other entries of $\mathcal{A}$ are zeros. Then $\mathcal{A}x^2$ consists of homogeneous polynomials in $x_1, x_2$ of degree $2$. Then the monomials in modified graded lexicographic order are $x_1^2,\; x_2^2,\; x_1x_2$. Then we can represent $\mathcal{A}x^2$ as
\[ \mathcal{A}x^2 =
\left( \begin{array}{c}
    x_1^2 \\
    x_2^2
\end{array} 
\right) =
\left( \begin{array}{ccc}
    1 & 0 & 0 \\
    0 & 1 & 0
\end{array} \right)
\left( \begin{array}{c}
    x_1^2 \\
    x_2^2 \\
    x_1x_2
\end{array} 
\right) =
\left( \begin{array}{ccc}
    1 & 0 & 0 \\
    0 & 1 & 0
\end{array} \right)
\left( \begin{array}{c}
    y_1 \\
    y_2 \\
    y_3
\end{array} 
\right)
\]
where $y_1 = x_1^2, \; y_2 = x_2^2,$ and $y_3 = x_1x_2 .$ Then $A=\left( \begin{array}{ccc}
    1 & 0 & 0 \\
    0 & 1 & 0
\end{array} \right)$ 
and for $y=(y_1, y_2, y_3)^T$ we have $\mathcal{A}x^2 = A y$.

\noindent Now we construct $\Bar{A}=\left( \begin{array}{ccc}
    1 & 0 & 0 \\
    0 & 1 & 0 \\
    0 & 0 & 0
\end{array} \right)$,
and for $q=(0,-1)^T$ we construct $\Bar{q}=(0, -1, 0)^T$.
Then solution set of $TCP(q,\mathcal{A})= SOL(q,\mathcal{A}) = \{ (0,1)^T \}$ and solution set of $LCP(\Bar{q},\Bar{A})= SOL(\Bar{q},\Bar{A}) = \{ (0,1, y_3)^T, y_3 \geq 0 \}$. Here for $x=(0,1)^T \in SOL(q,\mathcal{A}) $, we have $y=(0,1,0)^T \in SOL(\Bar{q},\Bar{A}) $.
\end{example}

\begin{example}
Let $\mathcal{A} \in T_{4,2}$ such that $a_{1111}=1, a_{1112}=-2, a_{1122}=1, a_{2222}=1 $ and all other entries of $\mathcal{A}$ are zeros. Then $\mathcal{A}x^3$ consist of homogeneous polynomials in $x_1, x_2 $ of degree $3$. Then the monomials in modified graded lexicographic order are $x_1^3,\; x_2^3,\; x_1^2 x_2,\; x_1 x_2^2$. Then we can represent $\mathcal{A}x^3$ as
\begin{align*}
    \mathcal{A}x^3 = & \left( \begin{array}{c}
    x_1^3  +  0\cdot x_2^3  -2x_1^2 x_2  +  x_1 x_2^2\\
    0\cdot x_1^3  +  x_2^3  +  0\cdot x_1^2 x_2  +  0\cdot x_1 x_2^2
\end{array} 
\right)\\
= &\left( \begin{array}{cccc}
    1 & 0 & -2 & 1 \\
    0 & 1 & 0 & 0
\end{array} \right)
\left( \begin{array}{c}
    x_1^3 \\
    x_2^3 \\
    x_1^2 x_2 \\
    x_1 x_2^3
\end{array} 
\right)\\
= & \left( \begin{array}{cccc}
    1 & 0 & -2 & 1 \\
    0 & 1 & 0 & 0
\end{array} \right)
\left( \begin{array}{c}
    y_1 \\
    y_2 \\
    y_3 \\
    y_4
\end{array} 
\right)
\end{align*}
where $y_1 = x_1^3, \; y_2 = x_2^3, \; y_3 = x_1^2 x_2, \; y_4=x_1 x_2^2 $. Let $A=\left( \begin{array}{cccc}
    1 & 0 & -2 & 1 \\
    0 & 1 & 0 & 0
\end{array} \right)$ 
then for $y=(y_1, y_2, y_3, y_4)^T$ we have $\mathcal{A}x^3 = A y$.

\noindent Now we construct $\Bar{A}=\left( \begin{array}{cccc}
    1 & 0 & -2 & 1 \\
    0 & 1 & 0 & 0 \\
    0 & 0 & 0 & 0 \\
    0 & 0 & 0 & 0
\end{array} \right)$,
and for $q=(0,-1)^T$ we construct $\Bar{q}=(0, -1, 0, 0)^T$.
Then solution set of the $TCP(q,\mathcal{A})= SOL(q,\mathcal{A}) = \{ (0,1)^T, (1,1)^T \}$ from (\cite{bai2016global}). The $LCP(\Bar{q},\Bar{A})$ is the quest to find $y\in \mathbb{R}^4$ such that $y\geq 0, \; \Bar{w}=\Bar{A} y +\Bar{q} \geq 0 $ and $y^T \Bar{w} =0 $,
where $\Bar{w}=\Bar{A} y +\Bar{q} = 
\left( \begin{array}{c}
    y_1 +0 -2y_3 +y_4   \\
    y_2 - 1 \\
    0 + 0 \\
    0 + 0
\end{array} \right)$ $= \left( \begin{array}{c}
    y_1 -2y_3 +y_4  \\
    y_2 - 1 \\
    0 \\
    0
\end{array} \right)$.

\noindent The solution set of $LCP(\Bar{q},\Bar{A})= SOL(\Bar{q},\Bar{A}) = \{ (2y_3-y_4, 1, y_3, y_4)^T, 2y_3 \geq y_4 \geq 0 \} \cup \{ (0,1,y_3,y_4) : y_4 \geq 2y_3 \geq 0 \} $. Here for $x=(0,1)^T \in SOL(q,\mathcal{A}) $, we have $y=(0,1,0,0)^T \in SOL(\Bar{q},\Bar{A}) $ and for $x=(1,1)^T \in SOL(q,\mathcal{A}) $, we have $y=(1,1,1,1)^T \in SOL(\Bar{q},\Bar{A}) $.
\end{example}

The following result establishes the $\omega$-uniqueness of $TCP(q,\mathcal{A}).$
\begin{theorem}\label{w uniqueness of tcp by lcp}
Let $LCP(\Bar{q},\Bar{A})$ be the auxiliary $LCP$ of  $TCP(q,\mathcal{A})$. If $LCP(\Bar{q},\Bar{A})$ has unique $w$-solution for all $\Bar{q} \in \mathbb{R}^N$ then $TCP(q,\mathcal{A})$ has unique $\omega$-solution (if exists) for all $q \in \mathbb{R}^n$.
\end{theorem}
\begin{proof}
Let $ SOL(\Bar{q},\Bar{A})$ be the solution set of $LCP(\Bar{q},\Bar{A})$. Then $ SOL(\Bar{q},\Bar{A}) = \{ y\in \mathbb{R}^N :\; y\geq 0, \; w= \Bar{A}y + \Bar{q} \geq 0, ~~ y^T w =0  \}$. If $LCP(\Bar{q},\Bar{A})$ has unique $w$-solution, then $\{ w= \Bar{A}y + \Bar{q} :y\in SOL(\Bar{q},\Bar{A})\} $ is singleton and $w_i$'s are uniquely determined for each $i\in [N]$. Now we show that the set $\{ \omega= \mathcal{A}x^{m-1} + q : x\in SOL(q,\mathcal{A}) \}$ is singleton. 
For any $x\in SOL(q,\mathcal{A}) $ by Theorem \ref{solution of LCP implies solution of TCP}, $\exists \; y \in SOL(\Bar{q},\Bar{A}) $ such that
\begin{align*}
    w_i = &  \left\{
\begin{array}{ll}
	 (\Bar{A}y + \Bar{q})_i   & ; \; \forall \; i\in [n] \\
	 0    & ; \; \forall \; i\in [N]\backslash [n]
	   \end{array}
 \right.\\
 = & \left\{
\begin{array}{ll}
	 (\mathcal{A}x^{m-1})_i + q_i   & ; \; \forall \; i\in [n] \\
	 0    & ; \; \forall \; i\in [N]\backslash [n]
	   \end{array}
 \right.\\
 = & \left\{
\begin{array}{ll}
	 \omega_i    & ; \; \forall \; i\in [n] \\
	 0    & ; \; \forall \; i\in [N]\backslash [n]
	   \end{array}
 \right. .
\end{align*}
Thus the values of $\omega_i$'s are exactly same as that of $w_i, \; \forall \; i \in [n]$. Since $w_i$'s are uniquely determined for each $i\in [N]$ and for all $y \in SOL(\Bar{q},\Bar{A})$, $\omega_i$'s are uniquely determined for each $i\in [n]$ and for all $x \in SOL(q,\mathcal{A}) $. Hence $\{ \omega= \mathcal{A}x^{m-1} + q : x\in SOL(q,\mathcal{A}) \}$ is singleton. Therefore $TCP(q,\mathcal{A})$ has unique $\omega$-solution.
\end{proof} 

Now we illustrate Theorem 3.8 with the help of the following example.

\begin{example}
Let $\mathcal{A} \in T_{3,2}$ such that $a_{111}=a_{222}=1 $ and all other entries of $\mathcal{A}$ are zeros. Then the corresponding $\Bar{A}$ is obtained as  $\Bar{A}=\left( \begin{array}{ccc}
    1 & 0 & 0 \\
    0 & 1 & 0 \\
    0 & 0 & 0
\end{array} \right)$.

\noindent Then for $q \geq 0$, $SOL(\Bar{q},\Bar{A}) = \{ (0,0,y_3)^T, y_3 \geq0 \}.$ Then the $w$-solution is uniquely determined with $w=(q_1, q_2, 0)^T.$

\noindent For $q=(q_1,q_2)^T \leq 0$, $SOL(\Bar{q},\Bar{A}) = \{ (-q_1, -q_2, y_3)^T, y_3 \geq 0 \}.$ Then the $w$-solution is uniquely determined with $w=(0, 0, 0)^T.$

\noindent For $q=(q_1,q_2)^T $ such that $q_1 >0, q_2 <0$, $SOL(\Bar{q},\Bar{A}) = \{ (0 ,-q_2 ,y_3)^T$, $y_3 \geq 0 \}.$ Then the $w$-solution is uniquely determined with $w=(q_1, 0, 0)^T.$

\noindent For $q=(q_1,q_2)^T $ such that $q_1 <0, q_2 >0$, $SOL(\Bar{q},\Bar{A}) = \{ (-q_1,0,y_3)^T$, $y_3 \geq 0 \}$. Then the $w$-solution is uniquely determined with $w=(0, q_2, 0)^T.$

\noindent Thus whatever the case may be, $w$-solution of $LCP(\Bar{q},\Bar{A})$ is uniquely determined.

\noindent Here the immediate conclusion is that $\omega$-solution of $TCP(q,\mathcal{A})$ is also unique.
\end{example}

We establish a sufficient condition for column adequate tensor.
\begin{theorem}\label{adequate matrix implies adequate tensor}
Let $\mathcal{A}\in T_{m,n}$ where $m$ is even.  If $\Bar{A}$ is a column adequate matrix then $\mathcal{A}$ is a column adequate tensor.
\end{theorem}
\begin{proof}
Let $\Bar{A}$ be column adequate matrix. Then for $y\in \mathbb{R}^N$, $y_i (\Bar{A} y)_i \leq 0, \; \forall \; i\in [N] ~\implies ~\Bar{A} y =0$. Let for some $x\in \mathbb{R}^n,\; x_i(\mathcal{A}x^{m-1})_i \leq 0, \; \forall \; i\in [n].$ This implies $ x_i^{m-1}(\mathcal{A}x^{m-1})_i \leq 0, \; \forall \; i\in [n] $, since $m$ is even. Now after constructing $y$ with the help of $x$, we obtain  $(\mathcal{A}x^{m-1})_i=(\Bar{A} y)_i  $ and $y_i (\Bar{A} y)_i = x_i^{m-1}(\mathcal{A}x^{m-1})_i, \; \forall \; i \in [n] $. Therefore $ x_i^{m-1}(\mathcal{A}x^{m-1})_i \leq 0 \iff y_i (\Bar{A} y)_i \leq 0, \; \forall \; i \in [n]$ and by the construction of $\Bar{A}$ we have $y_i (\Bar{A} y)_i= 0, \; \forall \; i \in [N] \backslash [n]$. Therefore we obtain $y_i (\Bar{A} y)_i \leq 0, \; \forall \; i \in [N] \implies \Bar{A} y=0$, since $\Bar{A}$ is column adequate matrix. Now $\Bar{A} y=0 \implies (\Bar{A} y)_i=0, \; \forall \; i \in [N] \implies (\mathcal{A}x^{m-1})_i=0, \; \forall \; i \in [n] $. Thus $  x_i(\mathcal{A}x^{m-1})_i \leq 0, \; \forall \; i\in [n] \implies \mathcal{A}x^{m-1}=0 $. Hence $\mathcal{A}$ is a column adequate tensor.
\end{proof}

\noindent Now we establish the condition for tensor $\mathcal{A}$ under which $\Bar{A}$ is always column adequate. Consider the majorization matrix $M(\mathcal{A})$ of $\mathcal{A}$. Then it is easy to observe that $A= (M(\mathcal{A}) \;\; B)$ and then $\Bar{A} =$ $ \left( \begin{array}{cc}
    M(\mathcal{A}) & B \\
    O & O
\end{array} \right)$.

\begin{theorem}\label{condition for adequateness of block matrix }
For an $ \mathcal{A} \in T_{m,n}$ let $\Bar{A} =$ $ \left( \begin{array}{cc}
    M(\mathcal{A}) & B \\
    O & O
\end{array} \right)$. Then $\Bar{A}$ is column adequate if and only if the following two conditions hold
\begin{enumerate}[label=(\alph*)]
    \item $M(\mathcal{A})$ is column adequate.
    \item $B= O$.
\end{enumerate}
\end{theorem}
\begin{proof}
Let $\Bar{A}$ be column adequate matrix. Then for $y \in \mathbb{R}^N,$
\[ y_i (\Bar{A} y)_i \leq 0, \; \forall \; i \in [N] \implies (\Bar{A} y)_i =0, \; \forall \; i \in [N].\]
If $B \neq O$ we consider the vector $y \in \mathbb{R}^N$ such that $y= \left( \begin{array}{c}
     0_{n,1}\\
     e_{(N-n),1}
\end{array} \right)$, where $e_{(N-n),1} = (1, \cdots, 1)^T $, then
$y_i (\Bar{A} y)_i =0, \; \forall \; i \in [N].$ However

$\Bar{A} y= \left( \begin{array}{cc}
                 M(\mathcal{A})_{n,n} & B_{N-n,N-n} \\
                 O_{n,n} & O_{N-n,N-n}
                  \end{array} \right) $
$\left( \begin{array}{c}
     0_{n,1}\\
     e_{N-n,1}
\end{array} \right)$ = 
$\left( \begin{array}{c}
     (Be)_{n,1}\\
    O_{N-n,1}
\end{array} \right) \neq O_{N,1}$.
This contradicts the fact that $\Bar{A} $ is column adequate matrix. Therefore $B= O$. 

\noindent Now $B=O$ gives $\Bar{A} =$ $ \left( \begin{array}{cc}
    M(\mathcal{A}) & O \\
    O & O
\end{array} \right) .$  We prove that $M(\mathcal{A})$ is a column adequate matrix. Consider $z\in \mathbb{R}^n$ for which $ z_i (M(\mathcal{A}) z)_i \leq 0, \; \forall \; i \in [n] $. We choose $y \in \mathbb{R}^N$ such that
$y_i= \left\{
\begin{array}{ll}
	  z_i &;\; \forall \; i\in [n] \\
	  0  &; \; \forall \; i\in [N]\backslash [n]
	   \end{array}
 \right. .$
Then $(\Bar{A} y)_i = \left\{
\begin{array}{ll}
	 (M(\mathcal{A}) z)_i &;\; \forall \; i\in [n] \\
	  0  &; \; \forall \; i\in [N]\backslash[n]
	   \end{array}
 \right. .$ 
Then $ y_i (\Bar{A} y)_i \leq 0, \; \forall \; i \in [N]. \implies (\Bar{A} y)_i =0, \; \forall \; i \in [N]$, since $\Bar{A}$ is column adequate matrix. This implies $ (M(\mathcal{A}) z)_i = 0 ,\; \forall \; i\in [n].$
Hence $M(\mathcal{A})$ is column adequate matrix.

Conversely, let $M(\mathcal{A})$ is column adequate matrix and $B= O.$ Then for $y \in \mathbb{R}^N$, we define $z\in \mathbb{R}^n$ such that $z_i=y_i, \; \forall \; i \in [n]$. Then  $(\Bar{A} y)_i = \left\{
\begin{array}{ll}
	 (M(\mathcal{A}) z)_i &;\; \forall \; i\in [n] \\
	  0  &; \; \forall \; i\in [N]\backslash[n]
	   \end{array}
 \right. .$ Then \begin{align*}
   y_i (\Bar{A} y)_i \leq 0, \; \forall \; i \in [N] &  \implies z_i (M(\mathcal{A}) z)_i \leq 0, \; \forall \; i \in [n]   \\
     & \implies M(\mathcal{A}) z= 0, \mbox{ since } M(\mathcal{A}) \mbox{ is column adequate matrix}.
\end{align*}
Therefore for $y \in \mathbb{R}^N, \; y_i (\Bar{A} y)_i \leq 0, \; \forall \; i \in [N] \implies \Bar{A} y=0.$ Hence $\Bar{A}$ is column adequate matrix.
\end{proof}

\begin{remark}
If $B \neq O$, then $\Bar{A} $ is not a column adequate matrix.
\end{remark}

\begin{lemma}\label{lemma for with y bar}
Let $\Bar{A} =$ $ \left( \begin{array}{cc}
    M(\mathcal{A}) & O \\
    O & O
\end{array} \right)$
and $y \in SOL(\Bar{q},\Bar{A})$ where $y=(y_1, \cdots y_n, y_{n+1}, \cdots y_N)$. Then $\Bar{y}=(y_1, \cdots y_n, 0, \cdots, 0 )$ $\in SOL(\Bar{q},\Bar{A})$.
\end{lemma}
\begin{proof}
Since $y \in SOL(\Bar{q},\Bar{A})$ then $y\geq 0, \; w =\Bar{A}y +\Bar{q} \geq 0 $ and $y^Tw =0$. Now by the construction of $\Bar{y}$ we have $\Bar{y}\geq 0$, $\Bar{A}y= \Bar{A}\Bar{y}$ and $\Bar{y}^Tw =\left\{
\begin{array}{ll}
	 (y^T w)_i &;\; \forall \; i\in [n] \\
	  0  &; \; \forall \; i\in [N]\backslash[n]
	   \end{array}
 \right. .$ Thus we obtain $\Bar{y} \geq 0, \; \Bar{A}\Bar{y} +\Bar{q}= w \geq 0 $ and $\Bar{y}^Tw =0$. Hence $\Bar{y}\in SOL(\Bar{q},\Bar{A}) $.
\end{proof}

Bellow we find an equivalent condition for column adequate tensor which ensures the $\omega$-uniqueness for TCP$(q,\mathcal{A}).$
\begin{theorem}\label{majorization adequate implies omega uniqueness of TCP}
Let $\mathcal{A} \in T_{m,n}$, where $m$ is even and the auxiliary matrix $\Bar{A}$ is of the form $\Bar{A} =$ $ \left( \begin{array}{cc}
    M(\mathcal{A}) & O \\
    O & O
\end{array} \right)$, then the following are equivalent 
\begin{enumerate}[label=(\alph*)]
    \item $M(\mathcal{A})$ is column adequate matrix.
    \item $\Bar{A}$ is column adequate matrix.
    \item $\mathcal{A}$ is column adequate tensor.
    \item $LCP(\Bar{q},\Bar{A})$ has $w$-unique solution.
    \item $TCP(q,\mathcal{A})$ has $\omega$-unique solution.
\end{enumerate}
\end{theorem}
\begin{proof}
\noindent $ (a) \iff (b) : $ By Theorem \ref{condition for adequateness of block matrix }.

\noindent $(b) \implies (c) : $ By Theorem \ref{adequate matrix implies adequate tensor}.

\noindent $(c)\implies (b) : $ Let $\mathcal{A}$ be a tensor such that the auxiliary matrix $\Bar{A}$ is of the form $\Bar{A} =$ $ \left( \begin{array}{cc}
    M(\mathcal{A}) & O \\
    O & O
\end{array} \right)$. Then for $i\in[n]$,
\begin{equation}\label{equation for main result}
    (\mathcal{A}x^{m-1})_i = (\Bar{A}y)_i = (M(\mathcal{A})\mathcal{I}_m x^{m-1})_i.
\end{equation}

\noindent Let $\mathcal{A}$ be column adequate tensor then $\forall \; i \in [n]$, $x^{m-1}_i(\mathcal{A}x^{m-1})_i \leq 0 \implies \mathcal{A}x^{m-1}=0$.
For auxiliary matrix $\Bar{A}$, $y_i(\Bar{A}y)_i=\left\{
\begin{array}{ll}
	  x^{m-1}_i(\mathcal{A}x^{m-1})_i \; ; & \forall \; i\in [n] \\
	  0 \; ;  & \forall \; i\in [N]\backslash [n]
	   \end{array}
 \right.$.
Then $y_i(\Bar{A}y)_i \leq 0,\; \forall \; i \in [N] \implies x^{m-1}_i(\mathcal{A}x^{m-1})_i \leq 0, \; \forall \; i \in [n]$. This implies $\mathcal{A}x^{m-1}=0 $, since $\mathcal{A}$ is column adequate tensor. This implies $ M(\mathcal{A})\mathcal{I}_n x^{m-1} = O \implies (\Bar{A}y)_i =0, \; \forall \; i\in [n]  $, by (\ref{equation for main result}). By the construction of $\Bar{A}$, for all $y \in \mathbb{R}^N$ we have $(\Bar{A}y)_i =0, \; \forall \; i\in [N]\backslash [n] $. Thus $y_i(\Bar{A}y)_i \leq 0,\; \forall \; i \in [N] \implies (\Bar{A}y)_i = O$. Hence $\Bar{A}$ is column adequate matrix.

\noindent $(b) \iff (d) : $ By Theorem \ref{matrix w uniqueness result}.

\noindent $(d) \implies (e) : $ By Theorem \ref{w uniqueness of tcp by lcp}.

\noindent $(e) \implies (d) : $
Let $TCP(q,\mathcal{A})$ has $\omega$-unique solution. We prove that $LCP(\Bar{q},\Bar{A})$ has $w$-unique solution. Consider the solution set of $LCP(\Bar{q},\Bar{A})$ as $SOL(\Bar{q},\Bar{A})$. Then for $y\in SOL(\Bar{q},\Bar{A}) $ we get $w=\Bar{A}y + \Bar{q}$. If possible let $w$ be not unique. Non-uniqueness of $w$ implies there exists $y^1, y^2 \in SOL(\Bar{q},\Bar{A}) $ such that $w^1 \neq w^2$, where $w^1= \Bar{A}y^1 + \Bar{q} $ and $w^2= \Bar{A}y^2 + \Bar{q} $. Since $w^1 \neq w^2$ and $w_i=0, \; \forall \; i \in [N] \backslash [n]$, then $\exists \; k \in [n]$ such that $w^1_k \neq w^2_k$. Since $y^1,y^2 \in SOL(\Bar{q},\Bar{A}) $ then by Lemma \ref{lemma for with y bar} $\Bar{y}^1,\Bar{y}^2 \in SOL(\Bar{q},\Bar{A}) $. Then we obtain two vector $x^1, \; x^2 \;\in \mathbb{R}^n$ such that $x^1_i = (\Bar{y}^{1^{[\frac{1}{m-1}]}})_i$ and $x^2_i = (\Bar{y}^{2^{[\frac{1}{m-1}]}})_i,$ $\forall \; i \in [n].$ Now, 
\begin{align*}
w^1_k & = (\Bar{A}y^1 + \Bar{q})_k \\ 
                & = (\Bar{A}\Bar{y}^1 + \Bar{q})_k\\
                & = (M(\mathcal{A})\Bar{y}^1 + q)_k\\
                & = (M(\mathcal{A})x^{1^{[m-1]}} + q)_k\\
                & = (\mathcal{A}x^{1^{m-1}} + q)_k
\end{align*}
and 
\begin{align*}
w^2_k & = (\Bar{A}y^2 + \Bar{q})_k \\ 
                & = (\Bar{A}\Bar{y}^2 + \Bar{q})_k\\
                & = (M(\mathcal{A})\Bar{y}^2 + q)_k\\
                & = (M(\mathcal{A})x^{2^{[m-1]}} + q)_k\\
                & = (\mathcal{A}x^{2^{m-1}} + q)_k
\end{align*}

\noindent Since $TCP(q,\mathcal{A})$ has $\omega$-unique solution, then $(\mathcal{A}x^{1^{m-1}} + q)_k = (\mathcal{A}x^{2^{m-1}} + q)_k $. This contradicts the fact that $w^1_k \neq w^2_k$. Therefore $w$ is unique.
\end{proof}

\begin{lemma}\label{row diagonal lemma}
Let $\mathcal{A}\in T_{m,n}$ be even order row diagonal. $\mathcal{A}$ is column adequate tensor if and only if $M(\mathcal{A})$ is column adequate matrix.
\end{lemma}
\begin{proof}
Let $\mathcal{A}$ be row diagonal. Then $\mathcal{A}=M(\mathcal{A})\mathcal{I}_m$. Then $\mathcal{A} x^{m-1} = M(\mathcal{A})\mathcal{I}_m x^{m-1}$ $= M(\mathcal{A}) y $, where $x=(x_1, x_2, ..., x_n)^T$ and $y= x^{[m-1]} = (x_1^{m-1},x_2^{m-1},..., x_n^{m-1})^T$. 

\noindent Let $\mathcal{A}$ be column adequate tensor, i.e. $x_i (\mathcal{A} x^{m-1})_i \leq 0$ $\forall \; i\in[n]$ $ \implies \mathcal{A}x^{m-1} = 0 $.
Suppose for some $y \in \mathbb{R}^n $ we have $y_i (M(\mathcal{A}) y)_i \leq 0.$ Since $m$ is even, $\exists \; x= y^{[\frac{1}{m-1}]} = (y_1^{\frac{1}{m-1}},\;  y_2^{\frac{1}{m-1}}, \; ...,\;  y_n^{\frac{1}{m-1}} )^T$ such that 
\[ y_i (M(\mathcal{A})y)_i = x_i^{m-1} (M(\mathcal{A}) \mathcal{I}_m x^{m-1})_i = x_i^{m-1} (\mathcal{A} x^{m-1})_i . \]
Thus $y_i (M(\mathcal{A})y)_i = x_i^{m-1} (\mathcal{A} x^{m-1})_i  \leq 0$  $\mbox{ } \forall \; i\in [n]$ $\implies \mathcal{A} x^{m-1}= 0 \implies M(\mathcal{A})y = 0 $.
\noindent Therefor, $M(\mathcal{A})$ is column adequate matrix.

Conversely, let $M(\mathcal{A})$ be column adequate matrix then
\begin{align*}
     x_i^{m-1} (\mathcal{A} x^{m-1})_i & = y_i (M(\mathcal{A})y)_i  \leq 0, \;\; \forall \; i\in [n]  \\
    &\implies M(\mathcal{A})y = 0 \\
    &\implies \mathcal{A} x^{m-1} = 0.
\end{align*}
 Therefore $\mathcal{A}$ is column adequate tensor.
\end{proof}

\begin{example}
Let $\mathcal{A} \in T_{4,2}$ such that $a_{1111}= a_{2222}=1, \; a_{1222}= a_{2111}=-1, \;a_{1121}=3, \;a_{1112}=-2, \;a_{1211}=-1 $ and all other entries of $\mathcal{A}$ are zeros. Then $\mathcal{A}x^3$ consists of homogeneous polynomials in $x_1,\; x_2$ of degree $3$. Then the monomials in modified graded lexicographic order are $x_1^3,\; x_2^3,\; x_1^2 x_2,\; x_1 x_2^2$. We can represent $\mathcal{A}x^3$ as 
\begin{align*}
    \mathcal{A}x^3 = &
\left( \begin{array}{c}
   1 \cdot x_1^3 + -1\cdot x_2^3  + 0 \cdot x_1^2 x_2  + 0 \cdot x_1 x_2^2\\
  -1 \cdot x_1^3 + 1 \cdot x_2^3  +  0 \cdot x_1^2 x_2  +  0 \cdot x_1 x_2^2
\end{array} 
\right)\\
= & \left( \begin{array}{cccc}
    1 & -1 & 0 & 0 \\
    -1 & 1 & 0 & 0
\end{array} \right)
\left( \begin{array}{c}
    x_1^3 \\
    x_2^3 \\
    x_1^2 x_2 \\
    x_1 x_2^3
\end{array} 
\right)\\
= & \left( \begin{array}{cccc}
    1 & -1 & 0 & 0 \\
    -1 & 1 & 0 & 0
\end{array} \right)
\left( \begin{array}{c}
    y_1 \\
    y_2 \\
    y_3 \\
    y_4
\end{array} 
\right)
\end{align*}
where $y_1 = x_1^3, \; y_2 = x_2^3, \; y_3 = x_1^2 x_2$ and $ y_4=x_1 x_2^2 $. Let $A=\left( \begin{array}{cccc}
    1 & -1 & 0 & 0 \\
    -1 & 1 & 0 & 0
\end{array} \right).$ 
For $y=(y_1, y_2, y_3, y_4)^T$ we have $\mathcal{A}x^3 = A y$.

\noindent Now we construct $\Bar{A}=\left( \begin{array}{cccc}
    1 & -1 & 0 & 0 \\
    -1 & 1 & 0 & 0 \\
    0 & 0 & 0 & 0 \\
    0 & 0 & 0 & 0
\end{array} \right)$.
Then $M(\mathcal{A})=\left( \begin{array}{cc}
    1 & -1  \\
    -1 & 1 
\end{array} \right)$ and $B=\left( \begin{array}{cc}
    0 & 0  \\
    0 & 0 
\end{array} \right)$. $\Bar{A}$ is column adequate matrix, since $M(\mathcal{A})$ is column adequate matrix. Therefore $\mathcal{A}$ is column adequate tensor and $TCP(q,\mathcal{A})$ has unique $\omega$-solution by Theorem \ref{majorization adequate implies omega uniqueness of TCP}.

%This tensor is PSD and competent and is not pseudomonotone plus.

\end{example}

Bellow we establish a sufficient condition for $\omega$-unique solution of $TCP(q, \mathcal{A}).$
\begin{theorem}
Let $\mathcal{A}\in T_{m,n}$ be even order row diagonal. Then the following are equivalent:
\begin{enumerate}[label=(\alph*)]
    \item $M(\mathcal{A})$ is a column adequate matrix.
    \item $\mathcal{A}$ is a  column adequate tensor.
    \item The $TCP(q,\mathcal{A})$ has the $\omega$-unique solution.
\end{enumerate}
\end{theorem}
\begin{proof}
$(a) \iff (b) : $ By Lemma \ref{row diagonal lemma}.

\noindent $(a)\implies (c) : $
Since $\mathcal{A}$ is row diagonal, we have $\mathcal{A} x^{m-1}= M(\mathcal{A}) x^{[m-1]}$. Then the $TCP(q, \mathcal{A})$ is equivalent to the following $LCP(q, M(\mathcal{A}))$
\[ y\geq 0, \;\; M(\mathcal{A})y +q \geq 0, \;\; y^T (M(\mathcal{A})y +q)^T =0, \]
where $y=x^{[m-1]} = (x_1^{m-1},\; x_2^{m-1},\; ...,\; x_n^{m-1})^T$. Since $ M(\mathcal{A})$ is column adequate matrix, $LCP(q, M(\mathcal{A}))$ has $w$-unique solution. Hence $TCP(q, \mathcal{A})$ has $\omega$-unique solution.

\noindent $(c)\implies (a) : $
Let $\mathcal{A}$ be row diagonal tensor and $TCP(q, \mathcal{A})$ has $\omega$-unique solution. This implies $LCP(q, M(\mathcal{A}))$ has $w$-unique solution. Then by Theorem \ref{matrix w uniqueness result} we obtain $M(\mathcal{A})$ is column adequate matrix.
\end{proof}

\section{Conclusion}
In this paper we introduce column adequate tensor in the context of tensor complementarity problem. We establish several tensor theoritic properties. We define auxiliary matrix of the tensor and propose auxiliary linear complementarity problem related to tensor complementarity problem. We study the $\omega$-solution of tensor complementarity problem based on the corresponding auxiliary linear complementarity problem. We show that if $\mathcal{A}\in T_{m,n}$ is even ordered row-diagonal column adequate tensor then $TCP(q, \mathcal{A})$ has $\omega$-unique solution. We illustrate the results by various examples.

\section{Acknowledgment}
The author A. Dutta is thankful to the Department of Science and technology, Govt. of India, INSPIRE Fellowship Scheme for financial support. The author R. Deb is thankful to the Council of Scientific $\&$ Industrial Research (CSIR), India, Junior Research Fellowship scheme for financial support.

\bibliographystyle{plain}
\bibliography{document}

\end{document}